\date{January 2022}

\documentclass[11pt,reqno]{amsart}
\usepackage{latexsym,amsmath,amsfonts,amscd,amssymb}
\usepackage{graphics,color}
\usepackage{enumitem}

\theoremstyle{plain}
\newtheorem{theorem}{Theorem}[section]
\newtheorem{corollary}[theorem]{Corollary}
\newtheorem{lemma}[theorem]{Lemma}
\newtheorem{proposition}[theorem]{Proposition}
\theoremstyle{definition}

\newtheorem{remark}[theorem]{Remark}

\newcommand{\G}{\Gamma}

\renewcommand{\S}{\Sigma}

\newcommand{\bk}{\mathbf{k}}
\newcommand{\bd}{\partial}

\newcommand{\cM}{\mathcal{M}}
\newcommand{\cI}{\mathcal{I}}

\newcommand{\cO}{\mathcal{O}}

\newcommand{\RR}{\mathbb{R}}
\newcommand{\CC}{\mathbb{C}}

\newcommand{\PP}{\mathbb{P}}

\newcommand{\ZZ}{\mathbb{Z}}

\newcommand{\x}{\times}

\newcommand{\la}{\langle}
\newcommand{\ra}{\rangle}
\newcommand{\frM}{{\frak M}}
\newcommand{\frX}{{\frak X}}

\DeclareMathOperator{\stab}{Stab}

\DeclareMathOperator{\Hom}{Hom}

\DeclareMathOperator{\Gr}{Gr}
\DeclareMathOperator{\GL}{GL}
\DeclareMathOperator{\PGL}{PGL}
\DeclareMathOperator{\SL}{SL}

\DeclareMathOperator{\Spec}{Spec}
\DeclareMathOperator{\tr}{tr}

\providecommand\sslash{\mathbin{/\mkern-5.5mu/}}

\title{Coordinate rings of some $\SL_2$-character varieties}

\subjclass[2020]{Primary: 14M35. Secondary: 20G05, 14D20}
\keywords{Characters, representations, trace of matrix}

\author[V. Mu\~{n}oz]{Vicente Mu\~{n}oz}
 \address{Departamento de \'Algebra, Geometr\'ia y Topolog\'ia, Facultad de Ciencias,
  Universidad de M\'alaga, Campus de Teatinos s/n, 29071 Málaga,  Spain}
  \email{vicente.munoz@uma.es}
 
\author[J. Mart\'{\i}n-Ovejero]{Jes\'us Mart\'{\i}n Ovejero}
 \address{Departamento de M\'atem\'aticas, Facultad de Ciencias, Universidad de Salamanca, Plaza de la Merced 1, 37008 Salamanca, Spain}
\email{lemurx@usal.es}

\begin{document}

\begin{abstract}
 We determine generators of the coordinate ring of 
 $\SL_2$-character varieties. In the case of the free group $F_3$ we obtain
 an explicit equation of the $\SL_2$-character variety.
 For free groups $F_k$ we find transcendental generators. Finally, for the case
 of the $2$-torus, we get an explicit equation of the $\SL_2$-character variety
 and use the description to compute their $E$-polynomials.
\end{abstract}

\maketitle

\section{Introduction}\label{sec:introduction}

Let $\Gamma$ be a finitely generated group and let $G$ be an algebraic group over an
algebraically closed field $k$. The character variety $\frX(\G,G)$ parametrizes isomorphism
classes of representations $\rho: \Gamma \to G$. Character varieties are rich objects that contain geometric information linking
distant areas in mathematics. An important instance is when we take $\Gamma_g = \pi_1(\Sigma_g)$
to be the fundamental group of the compact orientable surface of genus $g\geq 1$. In this case, these character varieties are
one of the three incarnations of the moduli space of Higgs bundles, as stated by the celebrated
non-abelian Hodge correspondence \cite{Corlette:1988,hitchin,SimpsonI}. For this reason, character varieties of surface groups
have been widely studied, particularly regarding the computation of some algebraic invariants like their
$E$-polynomial.

Character varieties also play a prominent role in the topology of $3$-manifolds, starting
with the foundational work of Culler and Shalen \cite{CS}, where the authors used 
algebro-geometric properties of $\SL_2(\CC)$-character varieties to provide new proofs of remarkable results, as Thurston theorem that
says that the space of hyperbolic structures on an acylindrical $3$-manifold is compact, or the Smith conjecture.
Character varieties of $3$-manifolds have been also used to study knots $K\subset S^3$, by analyzing 
the character variety associated to the fundamental group of their complement, $\Gamma_K = \pi_1(S^3-K)$. 
For instance, the geometry of these knot character varieties has been studied in 
\cite{Florentino-Nozad-Zamora:2019,Lawton,LM} for trivial links (i.e.\ when $\Gamma$ is a free group),
and in \cite{KM, Martin-Oller, Munoz, MP} for the torus knot, among others.

Fix $G=\SL_r$, the group of matrices of size $r$ with trivial determinant.
If we have a presentation $\G=\la x_1,\ldots, x_k | r_1,\ldots, r_s \ra$, then $\frX(\G,G)$ parametrizes 
$k$-tuples $(A_1,\ldots, A_k)$ of matrices in $G$ 
subject to the relations $r_j(A_1,\ldots, A_k)=I$.
In the same vein as the isomorphism class of a matrix $A$ is determined by the traces of
its powers, $\tr(A^i)$, $1\leq i \leq r-1$, the $k$-tuples $(A_1,\ldots, A_k)$ are determined
by the traces of suitable products of the matrices. In this paper, we focus on the group $\SL_2$ and work out many
identities with traces. This serves to find coordinates for the character variety $\frX(\G,\SL_2)$.
Note that there is a natural  embedding $\frX(\G,\SL_2)\subset \frX(F_k,\SL_2)$, where
$F_k$ is the free group of $k$ elements, and $k$ is the number of generators of $\G$.
Therefore it is natural to look initially to the case of the free group.

The structure of $\SL_2(\CC)$-character varieties of free groups has been well-understood 
for some time. There is a modern treatment in \cite{Goldman}, where historical references can be found.
We thank Sean Lawton for pointing this out to us. More generally, there is an effective algorithm to 
compute the coordinate ring of any $\SL_2(\CC)$-character variety for any finitely presentable group in \cite{ABL}.
Here we obtain in an alternative way some of these coordinate rings, and connect with results of \cite{LMN} about
their $E$-polynomials. We start with the following:

\begin{theorem}\label{thm:3-intro}
 Let $A_1,\ldots, A_k\in \SL_2$, then the character variety $\frX_k=\frX(F_k, \SL_2)$ is parametri\-zed by
 $T_{i_1\ldots i_p}:= t_{A_{i_1}\ldots A_{i_p}}$, $i_1<\ldots < i_p$ with $p=1,2,3$. 
\end{theorem}

For the situation of $k=3$, we can obtain an explicit equation. By Theorem \ref{thm:3-intro}, 
the coordinates of $\frX_3=\frX(F_3,\SL_2)$ are
$(x,y,z,u,v,w,P)= (t_A, t_B, t_C, t_{BC}, t_{AC}, t_{AB}, t_{ABC})$, where
$(A,B,C)\in \frX_3$. We have the following:

\begin{theorem} \label{thm:P2-intro} 
The character variety $\frX_3  \subset \bk^7$ is a hypersurface defined by the equation 
$P^2 =  (w z+ v y+ u x- x y z)P -x^2 -y^2 -z^2+u y z +v x z +w x y- u v w -u^2-v^2-w^2+4$.
\end{theorem}

Next, we look at the case of the character varieties of a compact orientable surface $\S_g$ of genus $g\geq 1$.
Its fundamental group is 
 $$
 \pi_1(\S_g )=\la a_1,b_1,\ldots, a_g, b_g | \, \prod_{i=1}^g [a_i,b_i]=1\ra.
 $$
Take a conjugacy class $[\xi]$ determined by an element $\xi\in \SL_2$, then we define as in \cite{LMN}
 \begin{equation}\label{eqn:Mxi}
  \begin{aligned}
  \cM_\xi &= \{(A_{1},B_{1},\ldots,A_{g},B_{g}) \in (\SL_2)^{2g} \,| \prod_{i=1}^{g}[A_{i},B_{i}]=\xi\} \sslash \stab(\xi) \\
 &= \{(A_{1},B_{1},\ldots,A_{g},B_{g}) \in (\SL_2)^{2g} \,| \prod_{i=1}^{g}[A_{i},B_{i}] \in [\xi]\} \sslash \SL_2 \, .
 \end{aligned}
 \end{equation}
 
There are five different types of conjugacy classes, namely $[I], [-I], [J_+],  [J_-]$ and $[\xi_t]$, where 
$J_\pm =\left( \begin{array}{cc} \pm1 & 0\\ 1 & \pm1\end{array}\right)$ are the Jordan types, and 
$\xi_t=\left( \begin{array}{cc} \lambda & 0\\ 0 & \lambda^{-1} \end{array}\right)$, $t=\lambda+\lambda^{-1}$, $\lambda\in \CC-\{0,\pm 1\}$,
are the diagonal types.
For $\bk=\CC$, these varieties have been studied in \cite{LMN} and the $E$-polynomials are computed in a series of papers
\cite{LMN, MMp,MM}. For $g=1$, we have the following result from \cite[Theorem 1.1]{LMN}.

\begin{theorem}\label{thm11-intro}
  For the $2$-torus $\S_1=T^2$, the $E$-polynomials of $\cM_\xi$ are as follows:
  \begin{eqnarray*}
    e(\cM_{I}) &=& q^2+1\,, \\
    e(\cM_{-I}) &=& 1\, ,\\
    e(\cM_{J_+}) &=&q^2-2q-3 \, , \\
    e(\cM_{J_-}) &=& q^2 +3q\, ,\\
    e(\cM_{\xi_t}) &=& q^2+4q+1\, ,
  \end{eqnarray*}
where $q=\, uv$, $e(\cM_\xi)\in \ZZ[u,v]$.
\end{theorem}

We look at the character varieties $\cM_\xi$ more closely by working out trace identities of commutators of
two matrices. First, for matrices $(A,B)$, we determine the equation for $(x,y,z)=(t_A,t_B,t_{AB})$,
 $$
  F(x,y,z)=t_{[A,B]}= x^2 +y^2+z^2 - xyz -2.
$$
which produces the character varieties:
 $$
 \frX_t= F^{-1}(t)=\{ (A,B)\in (\SL_2)^2 | \, \tr([A,B])=t\}\sslash \SL_2,
 $$
for $t\in \CC$. Then
 \begin{itemize}
 \item $\frX_t= \cM_{\xi_t}$, for $t\neq \pm 2$,
 \item $\frX_2=\cM_{I}\cup \cM_{J_+}$,
 \item $\frX_{-2}=\cM_{-I}\cup \cM_{J_-}$.
 \end{itemize}
   
We study the geometry of the character varieties $\frX_t$, and recover the results of Theorem \ref{thm11-intro}.
More specifically,

\begin{theorem} \label{thm11bis-intro}
Let $t\in\CC$. We have the following:
\begin{itemize}
 \item For $t\neq \pm2$,  the character variety $\frX_t\subset \CC^3$ is a smooth surface,
 and $e(\frX_t)=q^2+4q+1$.
 \item For $t=2$, the character variety $\frX_2\subset \CC^3$ has $4$ ordinary double points. We have $\cM_{J_+}\subset \cM_I$, 
 $\frX_2=\cM_{I}$, and $e(\frX_2)=q^2+1$.
 \item For $t=-2$, the character variety $\frX_{-2}\subset \CC^3$ has only one singular point which is an ordinary double 
 point. We have $\frX_{-2}=\cM_{-I}\sqcup \cM_{J_-}$, and $e(\frX_{-2})=q^2+3q+1$.
  \end{itemize}
 \end{theorem}

 \noindent {\bf Acknowledgements.}
 We thank Sean Lawton for useful comments and references.
 The first author is partially supported by Project MCI (Spain) PID2020-118452GB-I00. The second author is supported by the Spanish Government research project  PGC2018-099599-B-I00. The second author would like to express his gratitude to the University of Málaga,  for giving him the opportunity to carry out a research stay with the first author.

\section{Moduli of representations and character varieties}\label{sec:character}

Let $\G$ be a finitely presented group, and let $G<\GL_r$ be an algebraic group over an algebraic closed
field $\bk$.
A \textit{representation} of $\G$ in $G$ is a homomorphism $\rho: \G\to G$.
Consider a presentation $\G=\la x_1,\ldots, x_k | r_1,\ldots, r_s \ra$. Then $\rho$ is 
determined by the $k$-tuple $(A_1,\ldots, A_k)=(\rho(x_1),\ldots, \rho(x_k))$
subject to the relations $r_j(A_1,\ldots, A_k)=I$, $1\leq j \leq s$. The space
of representations is
 \begin{equation} \label{eqn:rep}
 R(\G,G) = \Hom(\G, G) = \{(A_1,\ldots, A_k) \in G^k \, | \,
 r_j(A_1,\ldots, A_k)=I, \,  1\leq j \leq s \}.
 \end{equation}
Therefore $R(\G,G)$ is an affine algebraic set.

We say that two representations $\rho$ and $\rho'$ are
equivalent if there exists $P\in G$ such that $\rho'(g)=P^{-1} \rho(g) P$,
for every $g\in G$. This corresponds to a change of basis in $\bk^r$.
Note that the action of $G$ descends to an action of the
projective group $\mathrm{P}G<\PGL_r$ on $R(\G,G)$. The moduli space of representations
is the GIT quotient
 $$
 \frM (\G,G) = R(\G,G) \sslash G \, .
 $$
Recall that by definition of GIT quotient for an affine variety, if we write
$R(\G,G)=\Spec \cO$, then $\frM (\G,G)=\Spec \cO\,{}^{G}$.

Suppose from now on that $G=\SL_r$.
A representation $\rho$ is \textit{reducible} if there exists some proper 
subspace $V\subset \bk^r$ such that for  all $g\in G$ we have 
$\rho(g)(V)\subset V$;  otherwise $\rho$ is
\textit{irreducible}. Note that if $\rho$ is reducible, then let  $V\subset \bk^r$
an invariant subspace, and consider a complement $\bk^r =V\oplus W$. 
Let $\rho_1=\rho|_V$ and $\rho_2$ the induced representation on the
quotient space $W=\bk^r/V$. Then we can write $\rho=\begin{pmatrix} \rho_1 & 0\\
f& \rho_2\end{pmatrix}$, where $f: \G \to \Hom(W,V)$. Take $P_t=
\begin{pmatrix} t\, I_V & 0\\ 0& I_W \end{pmatrix}$, where $k=\dim V$.
Then $P_t^{-1}\rho P_t=\begin{pmatrix} \rho_1 & 0\\
t f& \rho_2\end{pmatrix} \to \rho'=\begin{pmatrix} \rho_1 & 0\\
0& \rho_2\end{pmatrix}$, when $t\to 0$. Therefore $\rho$ and $\rho'$
define the same point in the quotient $\frM(\G,G)$. Repeating this, we can
substitute any representation $\rho$ by some $\tilde\rho=\bigoplus \rho_i$,
where all $\rho_i$ are irreducible representations. We call this process 
\emph{semi-simplification}, and $\tilde\rho$ a semisimple
representation; also $\rho$ and $\tilde\rho$ are called
S-equivalent. The space $\frM(\G,G)$ parametrizes semi-simple representations
\cite[Thm.~ 1.28]{LuMa}.

Given a representation $\rho: \G\to G$, we define its
\textit{character} as the map $\chi_\rho: \G\to \bk$,
$\chi_\rho(g)=\tr \rho (g)$. Note that two equivalent
representations $\rho$ and $\rho'$ have the same character.
There is a character map $\chi: R(\G,G)\to \bk^\G$, $\rho\mapsto
\chi_\rho$, whose image
 \begin{equation}\label{e:ch-variety}
  \frX (\G,G)=\chi(R(\G,G))
 \end{equation}
is called the \textit{character variety} of $\G$. Let us give
$\frX(\G,G)$ the structure of an algebraic variety. The traces $\chi_\rho$
span a subring $B\subset A$. Clearly $B\subset A^{G}$. As
$A$ is noetherian, we have that $B$ is a finitely generated $\bk$-algebra. Hence
there exists a collection $g_1,\ldots, g_a$ of elements of
$G$ such that $\chi_\rho$ is determined by $\chi_\rho(g_1),\ldots,
\chi_\rho(g_a)$, for any $\rho$. Such collection gives a map
 $$
  \bar\chi:R(\G,G)\to \bk^a\, , \qquad
  \bar\chi(\rho)=(\chi_\rho(g_1),\ldots, \chi_\rho(g_a))\, ,
 $$
and $\frX(\G,G)\cong \bar\chi(R(\G,G))$. This endows $\frX(\G,G)$ with
the structure of an algebraic variety, which is independent of the chosen collection.
The natural algebraic map  
 $$
 \frM(\G,G)\to \frX(\G,G)
 $$ 
is an  isomorphism (see Chapter 1 in \cite{LM}). This is the same as to say that 
$B=A^G$, that is, the ring of invariant polynomials is generated by characters.

\subsection{Hodge structures and $E$-polynomials} \label{subsec:e-poly}

Later we will need the notion of $E$-polynomial, which is an invariant of a complex algebraic variety
constructed as an Euler characteristic of its Hodge numbers. We introduce the basic definitions. Here the ground field is 
$\bk=\CC$. A pure Hodge structure of weight $k$ consists of a finite dimensional complex vector space
$H$ with a real structure, and a decomposition $H=\bigoplus_{k=p+q} H^{p,q}$
such that $H^{q,p}=\overline{H^{p,q}}$, the bar meaning complex conjugation on $H$.
A Hodge structure of weight $k$ gives rise to the so-called Hodge filtration, which is a descending filtration
$F^{p}=\bigoplus_{s\ge p}H^{s,k-s}$. We define $\Gr^{p}_{F}(H):=F^{p}/ F^{p+1}=H^{p,k-p}$.

A mixed Hodge structure consists of a finite dimensional complex vector space $H$ with a real structure,
an ascending (weight) filtration $\cdots \subset W_{k-1}\subset W_k \subset \cdots \subset H$
(defined over $\RR$) and a descending (Hodge) filtration $F$ such that $F$ induces a pure Hodge structure of weight $k$ on each $\Gr^{W}_{k}(H)=W_{k}/W_{k-1}$. We define $H^{p,q}:= \Gr^{p}_{F}\Gr^{W}_{p+q}(H)$ and write $h^{p,q}$ for the {\em Hodge number} $h^{p,q} :=\dim H^{p,q}$.

Let $Z$ be any quasi-projective algebraic variety (possibly non-smooth or non-compact). 
The cohomology groups $H^k(Z)$ and the cohomology groups with compact support  
$H^k_c(Z)$ are endowed with mixed Hodge structures \cite{De}. 
We define the {\em Hodge numbers} of $Z$ by
$h^{k,p,q}_{c}(Z)= h^{p,q}(H_{c}^k(Z))=\dim \Gr^{p}_{F}\Gr^{W}_{p+q}H^{k}_{c}(Z)$ .
The $E$-polynomial is defined as 
 $$
 e(Z):=\sum _{p,q,k} (-1)^{k}h^{k,p,q}_{c}(Z) u^{p}v^{q}.
 $$

The key property of Hodge-Deligne polynomials that permits their calculation is that they are additive for
stratifications of $Z$. If $Z$ is a complex algebraic variety and
$Z=\bigsqcup_{i=1}^{n}Z_{i}$, where all $Z_i$ are locally closed in $Z$, then $e(Z)=\sum_{i=1}^{n}e(Z_{i})$.
Also $e(X\x Y)=e(X)e(Y)$.

When $h_c^{k,p,q}=0$ for $p\neq q$, the polynomial $e(Z)$ depends only on the product $uv$.
This will happen in all the cases that we shall investigate here. In this situation, it is
conventional to use the variable $q=uv$. 
Basic cases are $e(\CC)=q$, $e(\CC^r)=q^r$, $e(\PP^r)=q^r+\ldots+ q^2 + q+1$.
%

\section{The character variety for free groups}\label{sec:free}

Now we focus on the case of a free group. Let  $\Gamma=F_k:=\la x_1,x_2,\ldots, x_k\ra$ 
be the free group generated by $k$ elements.
Then, the space of representations of $F_k$ in $\SL_r$ is just
 $$
\Hom (F_k,\SL_r) = (\SL_r)^k = \{(A_1,A_2,\ldots, A_k) | \, A_i\in \SL_r \},
 $$
the space of $k$-tuples of matrices in $\SL_r$. The moduli space of $k$-tuples of matrices up
to conjugation is 
 $$
 \frM(F_k, \SL_r) = (\SL_r)^k \sslash \SL_r\, .
 $$
 
 As we said in Section \ref{sec:character}, this is isomorphic to the character variety $\frX(F_k,\SL_r)$.
 This implies that there are finitely many $g_1,\ldots, g_a \in F_k$ such that 
 a character $\chi_\rho \in \frX(F_k,\SL_r)$ is determined by $\chi_\rho(g_1),\ldots,
\chi_\rho(g_a)$. Denote $\rho(x_i)=A_i\in \SL_r$, and also 
 $$
 \tr(A)=t_A
 $$ 
for the trace of a matrix. For an element $g_j= x_{i_{j1}}\ldots x_{i_{j\ell_j}}\in F_k$, we have
 $$
  \chi_\rho(g)= \tr \rho(g) = \tr (\rho( x_{i_{j1}}) \cdots \rho(x_{i_{j\ell_j}}) )
 = \tr (A_{i_{j1}} \cdots A_{i_{j\ell_j}})  = t_{A_{i_{j1}} \cdots A_{i_{j\ell_j}}}\, .
$$
This implies that $\frX(F_k,\SL_r)$ is parametrized by the above traces for $j=1,\ldots, a$, that is
\begin{eqnarray}\label{eqn:char-Fk}
  \frM(F_k, \SL_r) & \longrightarrow & \frX(F_k,\SL_r) \subset \bk^a  \nonumber \\
 (A_1,\ldots, A_k) &\mapsto & \big( \, t_{A_{i_{11}} \cdots A_{i_{1\ell_1}}}\, , \, \ldots\, ,\, t_{A_{i_{a1}} \cdots A_{i_{a\ell_a}}}\, \big)
 \end{eqnarray}
is a parametrization of the character variety.

\begin{proposition} \label{prop:0}
If $k\geq 2$, the dimension of the character variety $\frX(F_k,\SL_r)$  is 
 $$
 \dim \frX(F_k,\SL_r)=(r^2-1)(k-1). 
 $$ 
If $k=1$ then 
 $$
 \dim \frX(F_1,\SL_r)=(r-1).
 $$
\end{proposition}

\begin{proof}
 Let us assume that $k\geq 2$. The action of $\SL_r$ on irreducible representations has finite stabilizer, so the action 
has generic orbits of dimension $\dim \SL_r$. This means that
 $$
\dim \frX(F_k,\SL_r)= \dim\,  (\SL_r)^k - \dim \SL_r= (k-1) \dim \SL_r =(r^2-1)(k-1).
$$
On the other hand, if $k=1$, the character variety  
$\frX(F_1,\SL_r)= \SL_r \sslash \SL_r$ is canonically isomorphic to $\bk^{r-1}$, 
as it is proved in \cite{MP}, so it has dimension $r-1$.
\end{proof}

From now on we focus on the case of rank $2$, that is, the group $\SL_2$.
We want to determine how many traces are needed in (\ref{eqn:char-Fk}).
First, we demonstrate some useful matrix identities.

\begin{lemma} \label{lem:1}
Let $P,Q\in \SL_2$. Then the following holds
\begin{equation}\label{e:qp}
QP = (t_{PQ}-t_Pt_Q) I + t_P Q+t_Q P - PQ.
\end{equation}
\end{lemma}

\begin{proof}
First of all, let us recall that  $t_{A^{-1}}=t_A$ for every $A\in \SL_2$. On the other hand, 
the relation $t_{AB}=t_{BA}$ holds for every pair of square matrices $A,B$ of the same size. 
Given $A\in \SL_2$, the following relation is given by the characteristic polynomial of $A$,
\begin{equation}\label{e:square}
 A^2 =t_A A -I
\end{equation}
and therefore, the following holds
 \begin{equation}\label{e:inverse}
 A^{-1}= t_A I - A
\end{equation}
 By \eqref{e:inverse}, we can write $(PQ)^{-1}$ as
 $$
 Q^{-1}P^{-1}=(PQ)^{-1} = t_{PQ} I -PQ .
 $$
Applying \eqref{e:inverse} to $Q^{-1}$ and $P^{-1}$ on the above equation we obtain the following
 $$
(t_Q I - Q)(t_P I -P) =t_{PQ}I-PQ,
 $$
therefore $QP = (t_{PQ}-t_Pt_Q) I + t_P Q+t_Q P - PQ$, as required.
\end{proof}

\begin{proposition} \label{prop:2} Let $A,B,C,P,Q\in\SL_2$, then the following holds. 
\begin{enumerate}[label=(\roman*)]
\item $t_I=2$
\item $t_{AB}=t_{BA}$
\item $t_{A^2}=t_A^2-2$ 
\item $t_{ABAB}=t_{AB}^2-2$
\item $t_{PBAQ}= t_{PQ} t_{AB}-t_{PQ}t_At_B+ t_A t_{PBQ} +t_B t_{PAQ} - t_{PABQ}$
\item $t_{PA^2Q}= t_A t_{PAQ}- t_{PQ}$
\item  $t_{PA^{-1} Q}= t_A t_{PQ} - t_{PAQ}$
\item $t_{ABC}=t_{A}t_{BC}+t_Bt_{AC} +t_Ct_{AB}-t_{A}t_{B}t_C-t_{ACB}$
\end{enumerate}
 \end{proposition}

\begin{proof}
$(i)$ and $(ii)$ Immediate. 

\noindent $(iii)$ Since $A\in \SL_2$,  the result follows from Equation (\ref{e:square}) taking traces. 

\noindent $(iv)$ It follows from $(iii)$ by just observing that $ABAB=(AB)^2$. 

\noindent $(v)$ Use the formula of Lemma \ref{lem:1} multiplying in the left by $P$ and in the right by  $Q$, to get 
 \begin{equation}\label{e:comm}
 PBAQ = P( (t_{AB}-t_At_B) I + t_AB+t_B A - AB)Q
 \end{equation}
and take traces to obtain the sought formula.

\noindent $(vi)$ Start with Equation (\ref{e:square}), and multiply  in the left by $P$ and in the right by  $Q$, to get
 $$
 P A^2 Q= t_A PAQ - PQ\, .
 $$
Finally take traces we get the required formula.

\noindent $(vii)$ From Equation \eqref{e:inverse}, we get
 $$
PA^{-1} Q= t_A PQ - PAQ,
 $$
and take traces.

\noindent $(viii)$ In $(v)$, take $P=I$, $Q=C$ to get 
$t_{BAC}= t_{C} t_{AB}-t_{C}t_At_B+ t_A t_{BC} +t_B t_{AC} - t_{ABC}$,
which rewrites as the above.
\end{proof}

\begin{theorem}\label{thm:3}
 Let $A_1,\ldots, A_k\in \SL_2$, take a monomial $x= A_{i_1}^{r_1} \ldots A_{i_m}^{r_m}$, with $r_j \in \ZZ$ and $1\leq i_1,\ldots, i_m\leq k$. 
Then $t_x$ has a (polynomial) expression in terms of 
 $$
 T_{i_1\ldots i_p}:= t_{A_{i_1}\ldots A_{i_p}} , \qquad 1\leq i_1<\ldots < i_p \leq k.
 $$
Therefore the ring of functions of $\frX_k=\frX(F_k,\SL_2)=\Spec \cO_{\frX_k}$ is given as 
 $$
 \cO_{\frX_k} = \bk \big[ \{ T_{i_1\ldots i_p}\}_{  1\leq i_1<\ldots < i_p \leq k} \big]/\cI ,
 $$
for some ideal $\cI$ of relations.
\end{theorem}

\begin{proof}
If $r_j<0$, we use Proposition \ref{prop:2}$(vii)$ to write $t_x$ in terms of traces of monomials in which $r_j\geq 0$.
Now, if $r_j\geq 2$, we use Proposition \ref{prop:2}$(vi)$ to write $t_x$ in terms of traces of monomials in which $r_j$ is smaller. Repeating
we can reach an expression with $r_j=0,1$. Doing this for all indices, we finally get a polynomial expression in terms of 
$t_{A_{i_1}\ldots A_{i_p}}$,  $1\leq i_1,\ldots, i_p \leq k$, where $i_j \neq i_{j+1}$.
That is, in the monomial two consecutive matrices are distinct.

Now suppose that $i_j > i_{j+1}$. Then we use Proposition \ref{prop:2}$(v)$ to get an expression in which the traces appearing
have either less number of matrices, or $A_{i_j}, A_{i_{j+1}}$ are swapped. In the first case, we can work by induction on the
number of matrices involved to get to the result (note that the other operations do not increase the number of matrices in a
given monomial). In the second case, now we get an expression $PA_{i_{j+1}}A_{i_j}Q$ with  $i_{j+1}<i_j $. If now there are
two consecutive matrices repeated (that is, a square), we use  Proposition \ref{prop:2}$(vi)$ again. Otherwise, we have managed to
reorder two matrices. We can permute the matrices with this process until $i_1$ is the lowest index, so that $i_1< i_2,\ldots, i_p$.
We continue in this fashion until $i_1< i_2<\ldots < i_p$.
\end{proof}

The number of monomials of the form $A_{i_1}\ldots A_{i_{p}}$ with $1\leq i_1<\ldots<i_p\leq k$ described in Theorem \ref{thm:3} is
  \begin{equation*}
 \sum_{p=1}^k \binom{k}{p} = 2^k-1 \, .
\end{equation*}

\begin{corollary} \label{cor:k=2}
 For $k=2$, the character variety $\frX_2=\frX(F_2,\SL_2)$ is isomorphic to $\bk^3$, and it is parametrized by $(t_A,t_B,t_{AB})$,
 for $(A,B)\in \frX_2$.
\end{corollary}

\begin{proof}
 We have by Proposition \ref{prop:0}, that $\dim \frX(F_2,\SL_2)=3$. By Theorem \ref{thm:3}, the traces
$t_A,t_B,t_{AB}$ parametrize. Therefore the result.
\end{proof}

\section{Equation of the character variety $\frX(F_3,\SL_2)$}

Since a general algorithm to compute  the ideal $\mathcal{I}$ described in Theorem  \ref{thm:3} is unknown, let us start by 
looking at the free group generated by three elements $F_3$. The aim of this section is to study the character variety
 $$
 \frX_3=\frX(F_3,\SL_2)=\{(A,B,C) | \, A,B,C \in \SL_2\}  \sslash \SL_2\, .
 $$
 By Proposition \ref{prop:0}, we have that $\dim \frX(F_3,\SL_2)=6$. By Theorem \ref{thm:3}, the traces
 $$
 t_A, t_B, t_C, t_{AB},t_{AC},t_{BC},t_{ABC}
 $$
generate the ring of functions of  $\frX(F_3,\SL_2)$. These are $7$ variables, hence there is an embedding
 $$
 \frX(F_3,\SL_2) \subset \bk^7\, ,
$$
and the character variety is a hypersurface defined by a single equation. To find such equation, we work
as follows. For the sake of clarity, let us set the following variables
\begin{align}\label{e:new-variables}
  & x=t_A, \quad y =t_B,\quad z=t_C \nonumber\\
  & u=t_{BC}, \quad v =t_{AC},\quad w=t_{AB} \\
  & P=t_{ABC} \nonumber
  \end{align}  

Now we complete Theorem \ref{thm:P2-intro}.

\begin{theorem} \label{thm:P2} The character variety $\frX_3 \subset \bk^7$ is a hypersurface defined by the equation 
 $$
 P^2 =  (w z+ v y+ u x- x y z)P -x^2 -y^2 -z^2+u y z +v x z +w x y- u v w -u^2-v^2-w^2+4. 
 $$
\end{theorem}

\begin{proof}
Since $A,B,C\in\SL_2$, by Proposition \ref{prop:2}$(iv)$, the following holds
 $$t_{ABCABC}=t_{ABC}^2 -2$$ Then
\begin{align*}
  t_{ABCABC} =& \, t_{ABBC}(t_{AC}-t_At_C)+t_At_{ABCBC}+t_Ct_{ABABC}-t_{ABACBC}  \qquad\quad \text{(Prop. \ref{prop:2}}(v))\\
 =&\,  t_{ABBC}(t_{AC}-t_At_C)+t_A \big( t_{ABC}(t_{BC}-t_Bt_C)+t_Bt_{ABCC}+t_Ct_{ABBC}-t_{ABBCC}\big) \\
&+t_C\big( t_{ABC}(t_{AB}-t_At_B)+t_Bt_{AABC}+t_At_{ABBC}-t_{AABBC} \big) \\
&-\big(t_{ACBC}(t_{AB}-t_At_B)+t_Bt_{AACBC}+t_At_{ABCBC}-t_{AABCBC}\big)  \quad \text{(Prop.  \ref{prop:2}}(v)) \\
  =&\,  (t_Bt_{ABC}-t_{AC}) (t_{AC}-t_At_C)+t_A \big( t_{ABC}(t_{BC}-t_Bt_C)+t_B(t_Ct_{ABC}-t_{AB}) \\ & 
   \qquad +t_C(t_Bt_{ABC}-t_{AC})-(t_Bt_Ct_{ABC}-t_Bt_{AB}- t_Ct_{AC}+  t_A)\big) \\
 &+t_C\big( t_{ABC}(t_{AB}-t_At_B)+t_Bt_At_{ABC}-t_Bt_{BC}+t_A(t_Bt_{ABC}-t_{AC}) \\ 
 & \qquad - (t_At_Bt_{ABC}-t_At_{AC}-t_Bt_{BC}+t_C) \big) \\
&-\big(t_{ACBC}(t_{AB}-t_At_B)+t_Bt_At_{ACBC}-t_Bt_{CBC} \\ 
& \qquad +t_At_{ABCBC}-t_At_{ABCBC}
+t_{BCBC}\big)  \qquad\qquad\qquad\qquad\qquad\text{(Prop.  \ref{prop:2}}(vi)) \\
=& \,  -z^2-x^2 +v x z -v^2 + (w z+ v y+ u x- x y z)P\\ 
 &-\big(t_{ACBC}t_{AB}-t_Bt_{CBC} +t_{BCBC}\big)  \\ 
=& \,  -z^2-x^2 +v x z -v^2 +(w z+ v y+ u x- x y z)P\\ 
 &-\Big(  \big(t_{AC}(t_{BC}-t_Bt_C)+t_Ct_{ABC}+t_Bt_{ACC}-t_{ABCC}\big)   t_{AB} \\
 &\qquad -t_B \big(t_{C} (t_{BC}-t_Bt_C)+t_B t_{CC}+t_Ct_{BC}-t_{BCC}\big)  \qquad\qquad\quad \text{(Prop.  \ref{prop:2}}(v))\\
 &\qquad +  t_{BC}^2-2  \Big)  \qquad\qquad \qquad\qquad \qquad\qquad \qquad\qquad\qquad\qquad \,\,\text{(Prop.  \ref{prop:2}}(iv))\\
 =& \,  -z^2-x^2 +v x z -v^2 +(w z+ v y+ u x- x y z)P\\ 
 &-\Big(  \big(t_{AC}(t_{BC}-t_Bt_C)+t_Ct_{ABC}+t_B(t_{AC}t_C-t_A)-t_{ABC}t_C+t_{AB}\big)   t_{AB} \\
 &\qquad -t_B \big(t_{C} (t_{BC}-t_Bt_C)+t_B (t_{C}^2-2)+t_Ct_{BC}-t_{BC}t_C+t_B\big) \\
 &\qquad +   t_{BC}^2 -2  \Big)   \qquad\qquad \qquad\qquad \qquad\qquad\qquad\qquad\qquad\qquad \text{(Prop.  \ref{prop:2}}(vi))\\
 =& \,  -x^2 -y^2 -z^2+u y z +v x z +w x y- u v w \\ &\qquad -u^2-v^2-w^2+2+ (w z+ v y+ u x- x y z) P\, .  
\end{align*}

\end{proof}

Theorem \ref{thm:P2} can be rewritten as the following equality with traces for triples of
matrices $A,B,C\in \SL_2$,
  \begin{equation}\label{eqn:tABC}
 \begin{aligned}
  t_{ABC}^2 =&\,  (t_At_{BC}+t_Bt_{AC}+t_Ct_{AB}-t_At_Bt_C) t_{ABC} -t_A^2-t_B^2-t_C^2 \\
     &+t_At_Bt_{AB}+t_At_Ct_{AC}+t_Bt_Ct_{BC}-t_{AB}^2-t_{AC}^2-t_{BC}^2-t_{AB}t_{AC}t_{BC}+4 \, .
\end{aligned}
 \end{equation}

\begin{corollary} \label{cor:k=3}
The variety $\frX_3=\frX(F_3,\SL_2)$ is a ramified double cover of the plane $\bk^6$. The variables $t_A,t_B, t_C,t_{AB},t_{AC},t_{BC}$
are trascendental generators and the ring of functions
$\cO_{\frX_3}$ is a degree $2$ extension of $\bk[t_A,t_B, t_C,t_{AB},t_{AC},t_{BC}]$. 
\end{corollary}

By Theorem \ref{thm:P2}, $\frX_3$ is a double cover over $\bk^6$ ramified over $V(\Delta)$, where $\Delta=Y^2-4X$ is the
discriminant, with $X$, $Y$ defined as in (\ref{eqn:XY}). This is a sextic in $\bk^6$. The singularities of $\frX_3\subset \bk^7$ are
at the points $(x,y,z,u,v,w) \in V(\Delta)$, $P=\frac12 X$, which are singular points of $V(\Delta)$.
The singular locus of $\frX_3$ is determined in \cite{Florentino-Lawton:2012}, and it is equal to the 
reducible locus $\frX_3^{\mathrm{red}}$, which consists of representations $(A,B,C)$ which can be put, in a suitable basis as
 $$
  A=\left(\begin{array}{cc} \lambda & 0 \\  0 &\lambda^{-1}\end{array}\right), \quad
  B=\left(\begin{array}{cc} \mu & 0 \\  0 &\mu^{-1}\end{array}\right), \quad
  C=\left(\begin{array}{cc} \nu & 0 \\  0 &\nu^{-1}\end{array}\right).
  $$
This is a $3$-dimensional subspace. 
Equivalently $A,B,C$ are pairwise commuting, which by Equation (\ref{eqn:[A,B]}), it amounts to the equations
 $$
 x^2+y^2+w^2=xyw-2, \  x^2+z^2+v^2=xzv-2, \ y^2+z^2+u^2=yzu-2,  \ P=\frac12 X .
 $$

\section{Generators of the ring of  $\frX(F_k,\SL_2)$ for $k\geq 4$}


Now we give an expression for the trace of matrices that extends Theorem \ref{thm:3} to 
products of more than $4$ matrices.

\begin{theorem}\label{t:p:four-mat}
 Let $A,B,C,D\in\SL_2$. The trace of $ABCD$ can be expressed as a polynomial expression in terms of   $t_{A},t_{B},t_{C},t_{AB},t_{AC},t_{AD},t_{BC},t_{BD},t_{CD},t_{ABC},t_{ABD}, t_{ACD}$ and $t_{BCD}$. More specifically,
\begin{equation}\label{eqn:tABCD}
 \begin{aligned}
    t_{ABCD} = &\,    \frac12\Big( t_At_{BCD}+t_Bt_{ACD}+t_Ct_{ABD}+t_Dt_{ABC}+ t_{AD}t_{BC}- t_{AC}t_{BD}+   t_{AB}t_{CD}  \\ &
  -t_{AD}t_Bt_C -t_{BC}t_At_D-t_{AB}t_Ct_D -t_{CD}t_At_B +t_At_Bt_Ct_D \Big) \, . 
 \end{aligned}
 \end{equation}
\end{theorem}

\begin{proof}
 By Proposition \ref{prop:2}$(v)$, we have that
 $$\begin{aligned}
  t_{ABCD} &= t_{AD}(t_{BC}-t_Bt_C) +t_Bt_{ACD}+t_Ct_{ABD}-t_{ACBD} \\
 t_{ACBD} &=   t_{CBDA} = t_{CA}(t_{BD}-t_Bt_D) +t_Bt_{CDA}+t_Dt_{CBA}-t_{CDBA} \\
 t_{CDBA} &=   t_{DBAC} = t_{DC}(t_{BA}-t_Bt_A) +t_Bt_{DAC}+t_At_{DBC}-t_{DABC} 
 \end{aligned}
 $$
  Substituting the expression of each equation into the previous one, and using the cyclicity of the traces,
  namely $t_{DABC}=t_{ABCD}$, we obtain:
 $$
 \begin{aligned}
    t_{ABCD} = &\, t_{AD}(t_{BC}-t_Bt_C) +t_Bt_{ACD}+t_Ct_{ABD}- \big(t_{CA}(t_{BD}-t_Bt_D)   \\ 
    &+t_Bt_{CDA}+t_Dt_{CBA} \big)+ t_{DC}(t_{BA}-t_Bt_A) +t_Bt_{DAC}+t_At_{DBC}-t_{ABCD}  \, ,
 \end{aligned}
 $$
and hence
  $$
  \begin{aligned}
    t_{ABCD} = &\,  \frac12\Big( t_{AD}(t_{BC}-t_Bt_C) +t_Bt_{ACD}+t_Ct_{ABD}- \big(t_{AC}(t_{BD}-t_Bt_D)+t_Bt_{ACD} \\ 
    &+t_Dt_{ACB}\big) +   t_{CD}(t_{AB}-t_At_B) +t_Bt_{ACD}+t_At_{BCD} \Big)  \\
  = &\,  \frac12\Big( t_{AD}(t_{BC}-t_Bt_C) +t_Bt_{ACD}+t_Ct_{ABD}- t_{AC}(t_{BD}-t_Bt_D) - t_Bt_{ACD} \\ 
    &-t_D \big( t_At_{BC}+ t_Bt_{AC}+t_Ct_{AB}-t_At_Bt_C-t_{ABC}\big) 
    +   t_{CD}(t_{AB}-t_At_B) \\ &+t_Bt_{ACD}+t_At_{BCD} )      \Big) 
    \qquad\qquad\qquad  \qquad\qquad\qquad\qquad\qquad\qquad\text{(Prop. \ref{prop:2}$(viii)$)}  \\
    =&\,    \frac12\Big( t_At_{BCD}+t_Bt_{ACD}+t_Ct_{ABD}+t_Dt_{ABC}+
 t_{AD}t_{BC}- t_{AC}t_{BD}+   t_{AB}t_{CD} \\ &
-t_{AD}t_Bt_C -t_{BC}t_At_D-t_{AB}t_Ct_D -t_{CD}t_At_B +t_At_Bt_Ct_D \Big) \, .  \qquad \,\quad \text{(Simplifying)} 
 \end{aligned}
$$
\end{proof}

As a consequence, the ring of functions of $\frX_k=\frX(F_k,\SL_2)$ is generated by traces of the product of at most 
three matrices. This completes the proof of Theorem \ref{thm:3-intro}.
     
\begin{corollary}\label{cor:thm:3}
 Let $A_1,\ldots, A_k\in \SL_2$, take a monomial $x= A_{i_1}^{r_1} \ldots A_{i_m}^{r_m}$, with $1\leq i_1,\ldots, i_m\leq k$ and $r_j \in \ZZ$.
Then $t_x$ has a (polynomial) expression in terms of 
 $$
 T_{i_1\ldots i_p}:= t_{A_{i_1}\ldots A_{i_p}} ,  1\leq i_1<\ldots < i_p \leq k,
 $$
with $p\leq 3$. Therefore the ring 
 $$
 \cO_{\frX} = \bk \big[ \{T_{i_1\ldots i_p}\}_{1\leq i_1<\ldots < i_p \leq k, 1\leq p\leq 3} \big]/\cI\, ,
 $$
for some ideal $\cI$ of relations.
\end{corollary}

\begin{proof}
By Theorem \ref{thm:3}, $T_{i_1\ldots i_p}$ with $1\leq p\leq k$, gives generators of the ring of all $t_x$. Now by Theorem \ref{t:p:four-mat},
$t_{A_{i_1}\ldots A_{i_4}}$ is expressible in terms of all traces of one, two and three matrices among $A_{i_1},\ldots, A_{i_4}$.  
In general,  for $p\geq 4$, 
 $$
  t_{A_{i_1}\ldots A_{i_p}} =t_{A_{i_1}A_{i_2} A_{i_3} (A_{i_4}\cdots A_{i_p})} 
  $$
 is expressible in terms of the traces of products of  one, two and three matrices among $A_{i_1},A_{i_2}, A_{i_3}$, and $ Q:=A_{i_4}\cdots A_{i_p}$.
 These are traces of products of less than $p-1$ matrices. By induction, we get the result.
 \end{proof}
 
In virtue of Corollary \ref{cor:thm:3}, 
the number of generators of the ring $\cO_{\frX_k}$ is
 \begin{equation*}
 k +\binom{k}{2} +\binom{k}{3}\, .
 \end{equation*}

If we look at the case $k=4$, to parametize
 $$ 
  \frX_4=\frX(F_4,\SL_2)=\{(A,B,C,D) | \, A,B,C ,D \in \SL_2\} \sslash \SL_2, 
 $$
Corollary \ref{cor:thm:3} says that we need  the traces
 $$
 t_A, t_B, t_C, t_D, t_{AB},t_{AC},t_{AD}, t_{BC},t_{BD},t_{CD}, t_{ABC}, t_{ABD}, t_{ACD}, t_{BCD}
 $$
giving an embedding $\frX(F_4,\SL_2) \subset \bk^{14}$. By 
Proposition \ref{prop:0}, $\dim \frX(F_4,\SL_2)=9$, so $5$ of the above traces
are algebraically dependent of the other ones. Letting 
\begin{equation}\label{eqn:XY}
 \begin{aligned} 
 X(x,y,z,u,v,w) &:= w z+ v y+ u x- x y z, \\
 Y(x,y,z,u,v,w) &:= -x^2 -y^2 -z^2+u y z +v x z +w x y- u v w -u^2-v^2-w^2+4,
 \end{aligned}
 \end{equation}
we have an equation
\begin{equation}\label{eqn:XY2}
 t_{ABC}^2= X( t_A, t_B, t_C, t_{BC}, t_{AC},t_{AB}) \cdot t_{ABC}+ Y( t_A, t_B, t_C, t_{BC}, t_{AC},t_{AB}) ,
 \end{equation}
and similarly for the others $t_{ABD}, t_{ACD}, t_{BCD}$. This gives four algebraic dependent variables.

\begin{proposition} \label{p:tCD}
The trace $t_{CD}$ is algebraically dependent with 
  $t_A, t_B, t_C, t_D, t_{AB}$, $t_{AC},t_{BC}$, $t_{AD}$ and $t_{BD}$. Therefore
  $t_A, t_B, t_C, t_D, t_{AB},t_{AC},t_{BC}, t_{AD},t_{BD}$ are trascendental  
 generators of $\cO_{\frX_4}$.
\end{proposition}

\begin{proof}
 Clearly there is an algebraic dependence relation between all these variables, as
 the transcendental degree of the field that they generate is $9$. The variables
 $t_A,t_B,t_C,t_D$ are clearly algebraically independent. Therefore, there is one 
 of the other variables that depends algebraically on the rest. Permuting the order
 of the matrices, we can assume that it is $t_{CD}$.
 \end{proof}
 
 It is not easy to find out an explicit algebraic equation satisfied by $t_{CD}$ in Proposition
 \ref{p:tCD}. This can be done as follows.
Consider the element $t_{ABCD}=t_{(AB)CD}$ and apply Equation (\ref{eqn:tABC}) to get
 \begin{equation}\label{e:t.}
 t_{ABCD}^2= X(t_{AB},t_C,t_{D},t_{CD},t_{ABD},t_{ABC}) t_{ABCD}+ Y(t_{AB},t_C,t_{D},t_{CD},t_{ABD},t_{ABC}),
 \end{equation}
with the expressions $X,Y$ appearing in (\ref{eqn:XY}). Now use 
Equation (\ref{eqn:tABCD}) to susbstitute $t_{ABCD}$ in the above. This 
gives an equation involving   $t_A, \ldots,  t_D, t_{AB}, \ldots, t_{CD}$ and $t_{ABC},t_{ABD}, t_{ACD}$, $t_{BCD}$. 
Using Theorem \ref{thm:P2} we have algebraic equations for $t_{ABC},\ldots, t_{BCD}$ in terms of
the traces $t_A,t_B,t_C,t_D, t_{AB},t_{AC},t_{AD},t_{BC},t_{BD}, t_{CD}$. This will yield an equation involving
all required traces.
Note that we can also work out an equation like (\ref{e:t.}) for $t_{ABCD}=t_{A(BC)D}$ or $t_{ABCD}=t_{AB(CD)}$
or $t_{ABCD}=t_{(DA)BC}$. This can serve to eliminate $t_{ABCD}$.

\begin{corollary} \label{c:p:cor}
In $\frX_k=\frX(F_k,\SL_2)$, we have parameters for $(A_1,\ldots, A_k)$ given by 
$t_{A_i}$, $t_{A_iA_j}$, $t_{A_iA_jA_k}$, $i<j<k$. The parameters
$$
t_{A_1},t_{A_2}, t_{A_1A_2}, \qquad t_{A_j}, t_{A_1A_j}, t_{A_2A_j},  \quad  j\geq 3,
$$
are transcendental generators of $\cO_{\frX_k}$.
\end{corollary}

\begin{proof}
First, by Proposition \ref{prop:0} the dimension of $\frX_k$ is $3(k-1)$. Now $t_{A_1},t_{A_2}, t_{A_1A_2}$ generate
$\cO_{\frX_2}$ by Corollary \ref{cor:k=2}. For $k=3$, Corollary \ref{cor:k=3} says that
$t_{A_1},t_{A_2}, t_{A_3}, t_{A_1A_2}, t_{A_1A_3}$, $t_{A_2A_3}$ are transcendental generators of $\cO_{\frX_3}$.
For $k\geq 4$, we use Proposition \ref{p:tCD} applied to $(A_1,A_2,A_i,A_j)$ to get an algebraic equation 
of $t_{A_iA_j}$ in terms of 
$t_{A_1},t_{A_2}, t_{A_i}, t_{A_j}$, $t_{A_1A_2}$, $t_{A_1A_i}, t_{A_1A_j}, t_{A_2A_i},t_{A_2A_j}$.
Therefore the given set of traces are transcendental generators. There cannot be less than they are 
because $\dim \frX_k=3k-3$, which is the number of parameters in the list.
\end{proof}

\section{Character variety of the $2$-torus}

Now we are going to focus on the $2$-torus $T^2$ and the space of representations of its fundamental group
$\G=\pi_1(T^2)=\la x,y | \, [x,y]=1\ra$ in $\SL_2$. By the general description in (\ref{eqn:rep}), we have that 
the character variety of a finitely generated group embeds as a subvariety of the character variety of the
free group $F_k$, where $k$ is the number of generators of the group. In this situation,
 $$
 \frX_{T^2}= \frX(T^2, \SL_2) \subset \frX(F_2, \SL_2) =\bk^3\, ,
  $$
the last equality by Corollary \ref{cor:k=2}. Then $\frX_{T^2}$ is parametrized by $(t_A,t_B,t_{AB})$,
and there will be an equation describing this variety. To find it, we work out a relation for the trace
of a commutator.

\begin{lemma} \label{lem:t[A,B]}
For matrices $A,B\in \SL_2$, we have
 \begin{equation}\label{eqn:[A,B]}
 t_{[A,B]}= t_A^2+t_B^2+t_{AB}^2-t_At_Bt_{AB} -2.
 \end{equation}
\end{lemma}

\begin{proof}
 We compute
  \begin{equation}\label{eqn:ttt}
  \begin{aligned}
   [A,B] =& ABA^{-1}B^{-1}= AB(t_A I-A) B^{-1} \qquad \qquad \quad \qquad\qquad\qquad \qquad \textrm{(by Eqn.\ (\ref{e:inverse}))} \nonumber\\
    =& \,t_A A- ABAB^{-1} \nonumber\\
     =& \,t_A A - A \big( (t_{AB}-t_At_B)I + t_A B+ t_B A - AB) B^{-1} \qquad \qquad \qquad \textrm{(by Eqn.\ (\ref{e:comm}))} \nonumber\\
    =& \,t_A A-   (t_{AB}-t_At_B)AB^{-1} - t_A A- t_B A^2B^{-1} + A^2   \nonumber\\
    =& - (t_{AB}-t_At_B)A(t_B I-B) - t_B (t_A A-I)(t_B I-B) + t_A A-I  \qquad \textrm{(Eqn.\ (\ref{e:square}))} \nonumber\\
    =& - t_{AB}t_B A + t_At_B^2 A +t_{AB}AB-t_At_B AB    - t_At_B^2 A +t_B^2 I \\ &+ t_At_B AB-t_B B     + t_A A-I \nonumber\\
    =& - t_{AB}t_B A +t_{AB} AB  +t_B^2 I  -t_B B + t_A A-I \, .
   \end{aligned}
   \end{equation}
Taking traces,
 $$
 t_{[A,B]}= -t_{AB}t_Bt_A+t_{AB}^2+2t_B^2-t_B^2+t_A^2-2,
 $$
producing the result. 
\end{proof}

From now on, we fix the ground field $\bk=\CC$.
Take a conjugacy class $[\xi]=\SL_2\cdot \xi$ determined by an element $\xi\in \SL_2$ (the action of $\SL_2$ by conjugation). 
We have the \emph{twisted} moduli space of representations as defined in (\ref{eqn:Mxi}),
 $$
  \cM_\xi  = \{ (A,B)\in (\SL_2)^2 | \, [A,B]\in [\xi] \} \sslash \SL_2\, .
  $$
There are five different types of conjugacy classes. We have $[I], [-I], [J_+],  [J_-]$ and $[\xi_t]$, where 
$J_\pm =\left( \begin{array}{cc} \pm1 & 0\\ 1 & \pm1\end{array}\right)$ are the Jordan types, and 
$\xi_t=\left( \begin{array}{cc} \lambda & 0\\ 0 & \lambda^{-1} \end{array}\right)$, $t=\lambda+\lambda^{-1}$, $\lambda\in \CC-\{0,\pm 1\}$,
are the diagonal types. Consider the trace map
 $$
  \tr: \SL_2 \longrightarrow \CC
  $$
and note that 
 \begin{align*}
 \tr^{-1}(t) &=  [\xi_t], \qquad t\in \CC-\{\pm 2\}, \\
 \tr^{-1}(2) &= [I] \sqcup [J_+], \\
 \tr^{-1}-(2) &= [-I] \sqcup [J_-]. 
 \end{align*}

Using the variables $x=t_A$, $y=t_B$, $z=t_{AB}$, Lemma \ref{lem:t[A,B]} gives the function
 \begin{equation}\label{eqn:ch-torus}
  F(x,y,z)=t_{[A,B]}= x^2 +y^2+z^2 - xyz -2.
 \end{equation}
Then we have the following \emph{twisted character varieties}:
 $$
 \frX_t= F^{-1}(t)=\{ (A,B)\in \SL_2 | \, \tr([A,B])=t\}\sslash \SL_2,
 $$
for $t\in \CC$. Then
 \begin{itemize}
 \item $\frX_t= \cM_{\xi_t}$, for $t\neq \pm 2$,
 \item $\frX_2=\cM_{I}\cup \cM_{J_+}$,
 \item $\frX_{-2}=\cM_{-I}\cup \cM_{J_-}$.
 \end{itemize}
   
\begin{remark}
Note the symmetry of equation (\ref{eqn:ch-torus}). This is given by the change of generators
$(A,B) \mapsto (AB,B^{-1})$, that changes $(t_A,t_B,t_{AB}) \mapsto (t_{AB}, t_B,t_A)$.
\end{remark}

We study the geometry of the character varieties $\frX_t$, and recover results of Theorem \ref{thm11-intro}.
Now we prove Theorem \ref{thm11bis-intro}.

\begin{theorem} \label{thm11bis}
Let $t\in\CC$. We have the following:
\begin{itemize}
 \item For $t\neq \pm2$,  the character variety $\frX_t\subset \CC^3$ is a smooth surface. 
 The $E$-polynomial of $\frX_t$ is $e(\frX_t)=q^2+4q+1$.
 \item For $t=2$, the character variety $\frX_2\subset \CC^3$ has $4$ ordinary double points. $\cM_{J_+}\subset \cM_I$, 
 $\frX_2=\cM_{I}$, and $e(\frX_2)=q^2+1$.
 \item For $t=-2$, the character variety $\frX_{-2}\subset \CC^3$ has only one singular point which is an ordinary double 
 point. Now $\frX_{-2}=\cM_{-I}\sqcup \cM_{J_-}$, and $e(\frX_{-2})=q^2+3q+1$.
  \end{itemize}
 \end{theorem}
 
 \begin{proof}
We start analysing the singular points of $\{F(x,y,z)=t\}$. We compute the derivatives of $F$,
$$
 \left(\frac{\bd F}{\bd x},\frac{\bd F}{\bd y},\frac{\bd F}{\bd z}\right)=(2x-yz,2y-xz,2z-xy).
$$
For a singular point $2x=yz$, $2y=xz$, $2z=xy$.
From this, we get $x^2=y^2=z^2 =\frac12 xyz$ and hence $F=\frac12 xyz - 2$.
 \begin{itemize}
 \item For $t=-2$, we have $x^2=y^2=z^2=\frac12 xyz=0$, so there is 
 a singular point $(x,y,z)=(0,0,0)$. The leading
 term of $F$ is $x^2+y^2+z^2$, hence the point is an ordinary double point.
  \item For $t=2$, we have $x^2=y^2=z^2= \frac12  xyz=4$.
  Therefore, the singular points are $(2,2,2)$, $(2,-2,-2)$, $(-2,2,-2)$ and $(-2,-2,2)$.
  Let us focus on one of them, say $(2,2,2)$, then the Hessian of $F$ is
 $$
 H_F(2,2,2)=\left(\begin{array}{ccc} 2 & -z & -y \\ -z & 2 & -x \\ -y & -x & 2 \end{array}\right)\Bigg|_{(2,2,2)} =  
 \left(\begin{array}{ccc} 2 & -2 & -2 \\ -2 & 2 & -2 \\ -2 & -2 & 2 \end{array}\right),
$$
 which is non-degenerate, hence it is an ordinary double point. The other singular points are similar.
 \item For $t \neq \pm 2$, we have  $x^2=y^2=z^2=\frac12 xyz= 2+t\neq 0$. Then  
 $x=\pm y$, $x=\pm z$, hence $2x=yz =\pm x^2$, and $x=\pm 2$. This implies
  that $x^2=4=2+t$, hence $t=2$. So for $t\neq \pm 2$, the surface $\frX_t$ is smooth.
  \end{itemize}
  
To proceed, consider the completion of $V(F-t)\subset \CC^3$ in the projective space $\PP^3$. 
This is given by the homogeneous polynomial
 $$
 \hat F_t= x^2u+y^2u+z^2u-xyz-(2+t)u^3,
 $$
for projective coordinates $[x,y,z,u]$. We compute the derivatives 
$$
\left(\frac{\bd \hat F_t}{\bd x},\frac{\bd \hat F_t}{\bd y},\frac{\bd \hat F_t}{\bd z},\frac{\bd \hat F_t}{\bd u}\right)
=(2xu-yz,2yu-xz,2zu-xy,x^2+y^2+z^2-3(2+t)u^2),
$$
and look at a point at infinity, that is, $u=0$. Then the above derivatives reduce to $(-yz,-xz,-xy,x^2+y^2+z^2)$.
This cannot vanish, because for this to be zero, two coordinates should vanish, and hence $x^2+y^2+z^2\neq 0$.
This means that $V(\hat F_t)$ is smooth at the points at infinity.

Now take $t\neq \pm 2$. Then $V=V(\hat F_t)\subset \PP^3$ is a smooth surface of degree $3$. 
By \cite[Example 9.11]{Lewis}, we have the Hodge numbers of $V$ to be 
$h^{1,0}=h^{0,1}=0$, $h^{2,0}=h^{0,2}=0$ and $h^{1,1}=7$. Hence, the $E$-polynomial is $e(V)=q^2+7q+1$, where $q=uv$.
Now, the intersection $V_\infty:=V\cap \{u=0\} =\ell_1\cup \ell_2\cup \ell_3$, consists of $3$ lines, which has $E$-polynomial
$e(V_\infty)=\sum e(\ell_i)- \sum e(\ell_i\cap \ell_j)= 3(q+1)- 3= 3q$. Therefore
  $$
  e(\frX_t)=e(V)-e(V_\infty) = q^2+4q+1.
    $$

Denote $f:\SL_2^2 \to \SL_2$, $f(A,B)=[A,B]$. As in  \cite[Section 4]{LMN} we denote
$X_0=f^{-1}(I)$, $X_1=f^{-1}(-I)$, $X_2=f^{-1}([ J_+])$, $X_3=f^{-1}([J_-])$, so that 
$f^{-1}(\tr^{-1}(2))=X_0\sqcup X_2$ and $f^{-1}(\tr^{-1}(-2))=X_1\sqcup X_3$.
By \cite[Section 4.3]{LMN}, the representations of 
 $$
 \bar X_2= \left\{ (A,B) | \, [A,B]=J_+=\left(\begin{array}{cc} 1 & 0 \\ 1 & 1 \end{array}\right)\right\}
  $$
are of the form $A=\left(\begin{array}{cc} a & 0 \\ b & a^{-1}\end{array}\right)$ and 
$B=\left(\begin{array}{cc} x & 0 \\ y & x^{-1}\end{array}\right)$. Conjugating by
$\left(\begin{array}{cc} t & 0 \\ 0 & 1 \end{array}\right)$, we get the matrices
 $$
 A_t=\left(\begin{array}{cc} a & 0 \\ tb & a^{-1}\end{array}\right),
 B_t=\left(\begin{array}{cc} x & 0 \\ ty & x^{-1}\end{array}\right), [A_t,B_t]=J_{t,+} =
 \left(\begin{array}{cc} 1 & 0 \\ t & 1 \end{array}\right).
  $$
Taking $t\to 0$, we get in the limit representations in $X_0$.
Therefore $X_2$ is contained in the closure of $X_0$. This implies that $\cM_{J_+} \subset \cM_I$.
Hence $\frX_2=\cM_I$.

To compute the $E$-polynomial note that $X=V(\hat F_2)\subset  \PP^3$ appears as a degeneration of $V(\hat F_t)$
 when $t \to 2$. Such degeneration produces $4$ singularities which are ordinary double points.
 Each of them reduces the Betti number $b_2$ by one, hence $b_2(X)=3$. Therefore
 the $E$-polynomial of $X$ is $e(X)=q^2+3q+1$. Now
  $$
   e(\frX_2)=e(X)-e(V_\infty)= q^2+1.
   $$
   
Finally, we look at the case $t=-2$. By \cite[Section 4.2]{LMN}, $\cM_{-I}$ consists of one point, which
has representative $A= \left(\begin{array}{cc} i & 0 \\ 0 & -i\end{array}\right)$,
 $B=\left(\begin{array}{cc} 0 & 1 \\ -1 & 0\end{array}\right)$, with $(t_A,t_B,t_{AB})=(0,0,0)$, the singular
 point of $\frX_{-2}$.
Regarding $\cM_{J_-}$, according to \cite[Section 4.4]{LMN} the matrices of 
 $$
 \bar X_3= \left\{ (A,B) | \, [A,B]=J_-=\left(\begin{array}{cc} -1 & 0 \\ 1 & -1 \end{array}\right)\right\}
  $$
are of the form $A=\left(\begin{array}{cc} a & c \\ b & d\end{array}\right)$ and 
$B=\left(\begin{array}{cc} x & z \\ y & w\end{array}\right)$ with $z=2(x+w)$, $c=-2(a+d)$, $cy+2dw+bz=0$.
This implies that it cannot be $(t_A,t_B)=(z,c)=(0,0)$. So $\cM_{J_-}=\frX_{-2}-\cM_{-I}$, and hence
$\frX_{-2}=\cM_{J_-}\sqcup \cM_{-I}$. 

To compute the $E$-polynomial note that $Y=V(\hat F_{-2})\subset  \PP^3$ appears as a degeneration of $V(\hat F_t)$
 when $t \to -2$. Such degeneration produces one ordinary double point, hence $b_2(Y)=6$ and
 $e(Y)=q^2+6q+1$. Thus
   $$
   e(\frX_{-2})=e(Y)-e(V_\infty)= q^2+3q+1 = e(\cM_{J_-}) + e(\cM_{-I}).
   $$
 \end{proof}

Note that the results of the $E$-polynomials of Theorem \ref{thm11bis} agree with those of Theorem \ref{thm11-intro}. 

\section{Character variety of the genus $2$ surface}

We look at the character variety
 \begin{eqnarray*}
  \cM_2 &=& \frX(\pi_1(\S_2),\SL_2) = \{(A,B,C,D)  \in (\SL_2)^{4} \,|\,  [A,B]\,[C,D]=I\} \sslash \SL_2 \, .
 \end{eqnarray*}
The dimension of $\cM_g=\frX(\pi_1(\S_g),\SL_2)$, for the orientable compact surface $\S_g$ of genus $g$,
 is $\dim \cM_g=6g-6$. Therefore $\dim \cM_2=6$.

\begin{proposition}
 The ring $\cO_{\cM_2}$ has trascendental generators $t_A,t_B,t_C,t_D,t_{AB},t_{AC}$.
\end{proposition}

\begin{proof}
We have that $\cM_2 \subset \frX_4= \frX(F_4,\SL_2)$ and $\dim \frX_4=9$, generated by
$t_A,t_B,t_C,t_D$, $t_{AB},t_{AC},t_{BC}, t_{AD}, t_{BD}$, where $t_{CD}$ is algebraically
dependent on the previous ones, by Corollary \ref{c:p:cor}.
Using $[A,B]=[C,D]^{-1}$, we have $t_{[A,B]}=t_{[C,D]}$, and using Equation (\ref{eqn:[A,B]}),
 $$
 t_A^2+t_B^2+t_{AB}^2-t_At_Bt_{AB} =  t_C^2+t_D^2+t_{CD}^2-t_Ct_Dt_{CD},
$$
thereby $t_{CD}$ is algebraically dependent on $t_A,t_B,t_C,t_D,t_{AB}$.
  
Using Equation (\ref{eqn:ttt}), we get
 $$
  t_{[A,B]C} =-t_{AB}t_Bt_{AC}+t_{AB}t_{ABC} +t_B^2t_C-t_Bt_{BC}+t_At_{AC} -t_C
  $$
From $[A,B]\, [C,D]=I$, we rewrite $[A,B]C =DCD^{-1}$ which implies that
 $$
  -t_{AB}t_Bt_{AC}+t_{AB}t_{ABC} +t_B^2t_C-t_Bt_{BC}+t_At_{AC} -t_C= t_C
  $$
Using (\ref{eqn:XY2}), which is an algebraic dependence of $t_{ABC}$ on 
$t_A,t_B,t_C,t_{AB},t_{AC},t_{BC}$, we get unravelling the above an algebraic equation, 
and then we isolate $t_{BC}$ as algebraic dependent of
$t_A,t_B,t_C,t_{AB},t_{AC}$.

Now we use the equation $[A,B]\, [C,D]=I$, to get
$[C,D]A=BAB^{-1}$, and working as before we get an algebraic dependence of $t_{AD}$
on $t_A,t_C,t_D,t_{CD},t_{AC}$. But using the dependence of $t_{CD}$ on $t_{AB}$, we get
that $t_{AD}$ is algebraically dependent of 
$t_A,t_B,t_C,t_D, t_{AB},t_{AC}$.

Finally take $[A,B]\, [C,D]=I$, to rewrite 
$D^{-1}[A,B]=CD^{-1}C^{-1}$,
and work analogously to get an algebraic dependence of $t_{BD}$ in terms of 
$t_A,t_B,t_D,t_{AB},t_{AD}$. Using the previous paragraph, we get that
$t_{BD}$ is algebraic dependent on 
$t_A,t_B,t_C,t_D,t_{AB},t_{AC}$. 
\end{proof}


\begin{thebibliography}{11}


\bibitem{ABL}
C. Ashley, J-P. Burelle and S. Lawton, \textit{Rank 1 character varieties of finitely presented groups},
Geom. Dedicata {\bf 192} (2017) 1--19.

\bibitem{Corlette:1988}
K.~Corlette, \textit{Flat {$G$}-bundles with canonical metrics}, {J. Diff. Geom.} {\bf 28} (1988) 361--382.

\bibitem{CS} M. Culler and P. Shalen, \emph{Varieties of group representations and
splitting of $3$-manifolds}, {Annals Math. (2)} \textbf{117} (1983) 109--146.


\bibitem{De} P. Deligne,
Th\'eorie de Hodge II, Publ. Math. I.H.E.S. 40 (1971) 5--57.

\bibitem{Florentino-Lawton:2012}
C. Florentino and S. Lawton, \emph{Singularities of free group character varieties}, {Pac. J. Math. (1)} {\bf 260} (2012) 149--179.

\bibitem{Florentino-Nozad-Zamora:2019}
C. Florentino, A. Nozad and A. Zamora, \textit{Serre polynomials of $\SL_n$- and $\PGL_n$-character varieties of free groups}, 
 J. Geom. Physics \textbf{161} (2020) 104008.

%

%
%
%
%

\bibitem{Goldman} 
W. Goldman, \textit{Trace Coordinates on Fricke spaces of some simple hyperbolic surfaces},
Handbook of Teichmüller theory. Vol. II, 611--684, IRMA Lect. Math. Theor. Phys., 13, Eur. Math. Soc., Zürich, 2009.


\bibitem{hitchin}
N. Hitchin, \emph{The self-duality equations on a Riemann surface}, Proc. London Math. Soc. (3) \textbf{55} (1987) 59--126.

\bibitem{KM}   
T. Kitano and T. Morifuji, \textit{Twisted Alexander polynomials for irreducible $\SL(2,\CC)$-representations of torus knots}, Ann. Sc. Norm. Super. Pisa Cl. Sci. (5) \textbf{11} (2012) 395--406.

\bibitem{Lawton} S. Lawton,
\textit{Minimal affine coordinates for $\SL(3, \CC)$-character varieties of free groups},
J. Algebra \textbf{320} (2008) 3773--3810.

\bibitem{LM} S. Lawton and  V. Mu\~noz,
\emph{E-polynomial of the  $\SL(3, \CC)$-character variety of free groups}, Pac. J. Math. \textbf{282} (2016) 173--202.


\bibitem{Lewis} J. Lewis, 
\emph{A Survey of the Hodge Conjecture}, CRM Monograph Series, Vol.\ 10, AMS, 1999, 2nd  Edition.

\bibitem{LMN} M. Logares, V. Mu\~noz and P.E. Newstead, 
\emph{Hodge polynomials of $\SL(2,\mathbb{C})$-character varieties for curves of small genus}, 
Rev. Mat. Complut. \textbf{26} (2013) 635--703. 

\bibitem{LuMa}
{A.~Lubotzky and A.~Magid},  
\textsl{Varieties of representations of finitely generated groups}, Mem. Amer. Math. Soc. \textbf{58} (1985). 

\bibitem{Martin-Oller}
J. Mart\'{\i}n-Morales and A-M. Oller-Marc\'en, 
\emph{On the varieties of representations and characters of a family of one-relator subgroups},
Topol. Appl. \textbf{156} (2009) 2376--2389.


\bibitem{MMp} J. Martínez and V. Mu\~noz, \emph{E-polynomial of the SL(2,C)-character variety of a complex curve of genus 3},
Osaka J. Math. \textbf{53} (2016) 645--681.

\bibitem{MM} J. Martínez and V. Mu\~noz, \emph{E-polynomials of the $\SL(2,\CC)$-character varieties of surface groups},
Internat.\ Math.\ Research Notices \textbf{2016} (2016) 926--961.

\bibitem{Munoz} V. Mu\~noz, \textsl{The $\SL(2,\CC)$-character varieties of torus knots},
Rev. Mat. Complut. \textbf{22} (2009) 489--497.

\bibitem{MP} V. Mu\~noz and J. Porti, \textsl{Geometry of the $\SL(3,\CC)$-character variety of torus knots},
Algebraic Geometric Topology \textbf{16} (2016) 397--426.


\bibitem{SimpsonI}
C.~ Simpson, \textit{Moduli of representations of the fundamental group of a 
smooth projective variety. {I}}, {Inst. Hautes \'Etudes Sci. Publ. Math.} {\bf 79} (1994) 47--129.

  
\end{thebibliography}
\end{document}